\documentclass[11pt,draft]{article}
\title{A note on the abelianizations of finite-index subgroups of the mapping class group}
\author{Andrew Putman}
\usepackage{amsmath}
\usepackage{amssymb}
\usepackage{amsthm}
\usepackage{epsfig}
\usepackage[margin=1.25in]{geometry}
\usepackage{amsfonts}
\usepackage[font=small,format=plain,labelfont=bf,up,textfont=it,up]{caption}
\usepackage{amscd}
\usepackage{pinlabel}
\usepackage{stmaryrd}
\usepackage{type1cm}
\usepackage{calc}
\usepackage{mathptmx}  

\theoremstyle{plain}
\newtheorem{theorem}{Theorem}[section]
\newtheorem{maintheorem}{Theorem}
\newtheorem{maincorollary}[maintheorem]{Corollary}

\newtheorem{lemma}[theorem]{Lemma}

\theoremstyle{definition}

\theoremstyle{remark}
\newtheorem*{remark}{Remark}

\DeclareMathOperator{\Mod}{Mod}
\newcommand\JKer{\text{${\mathcal K}$}}
\newcommand\Torelli{\text{${\mathcal I}$}}
\DeclareMathOperator{\Sp}{Sp}

\newcommand\R{\text{$\mathbb{R}$}}

\newcommand\Z{\text{$\mathbb{Z}$}}

\DeclareMathOperator{\HH}{H}

\newcommand\Span[1]{\text{$\langle #1 \rangle$}}

\newcommand\Figure[3]{
\begin{figure}[t]
\centering
\centerline{\psfig{file=#2,scale=60}}
\caption{#3}
\label{#1}
\end{figure}}

\begin{document}

\maketitle

\section{Introduction}

Let $\Sigma_{g,b}^p$ be an oriented genus $g$ surface with $b$ boundary components and $p$ punctures and
let $\Mod(\Sigma_{g,b}^p)$ be its {\em mapping class group}, that is, the group of isotopy classes of
orientation--preserving diffeomorphisms of $\Sigma_{g,b}^p$ that fix the boundary components and punctures
pointwise (we will omit $b$ or $p$ when they are zero).  
A long--standing conjecture of Ivanov (see \cite{IvanovProblems} for a recent
discussion) says that for $g \geq 3$, the
group $\Mod(\Sigma_{g,b}^p)$ does not virtually surject onto $\Z$.  In other words, if $\Gamma$ is
a finite-index subgroup of $\Mod(\Sigma_{g,b}^p)$, then $\HH_1(\Gamma;\R) = 0$.

The goal of this note is to offer some evidence for this conjecture.  If $G$ is a group and
$g \in G$, then we will denote by $[g]_G$ the corresponding element of $\HH_1(G;\R)$.  Also, for
a simple closed curve $\gamma$ on $\Sigma_{g,b}^p$, we will denote by $T_{\gamma}$ the corresponding
right Dehn twist.  Observe that if $\Gamma$ is any finite-index subgroup of $\Mod_{g,b}^p$, then
$T_{\gamma}^n \in \Mod_{g,b}^p$ for some $n \geq 1$.  Our first result is the following.

\begin{maintheorem}[Powers of twists vanish]
\label{theorem:dehntwistsvanish}
For some $g \geq 3$, let $\Gamma<\Mod(\Sigma_{g,b}^p)$ satisfy $[\Mod(\Sigma_{g,b}^p):\Gamma]<\infty$
and let $\gamma$ be a simple closed curve on $\Sigma_{g,b}^p$.  Pick
$n \geq 1$ such that $T_{\gamma}^n \in \Gamma$.  Then $[T_{\gamma}^n]_{\Gamma} = 0$.
\end{maintheorem}

\begin{remark}
After this paper was written, Bridson informed us that in
unpublished work, he had proven a result about mapping class group actions on CAT(0) spaces 
that implies Theorem \ref{theorem:dehntwistsvanish}.  Bridson's work will appear
in \cite{Bridson}.
\end{remark}

We use this to verify Ivanov's conjecture for a class of examples.
For a long time, the only positive evidence for Ivanov's conjecture was a result of Hain \cite{HainTorelli} that 
says that it holds for all finite--index subgroups containing the {\em Torelli group} $\Torelli_{g,b}^p$, that is,
the kernel of the action of $\Mod(\Sigma_{g,b}^p)$ on $\HH_1(\Sigma_g;\Z)$ induced by filling in all
the punctures and boundary components.  The group $\Torelli_{g,b}^p$ contains the
{\em Johnson kernel} $\JKer_{g,b}^p$, which is the subgroup generated by Dehn
twists about separating curves.  A result of Johnson \cite{JohnsonHomo} says that
$\JKer_{g,b}^p$ is an infinite-index subgroup of $\Torelli_{g,b}^p$.

For a subgroup $\Gamma$
of $\Mod(\Sigma_{g,b}^p)$, denote by $K(\Gamma)$ the subgroup of $\Gamma \cap \JKer_{g,b}^p$
generated by the set
$$\{\text{$T_{\gamma}^n$ $|$ $\gamma$ a separating curve, $n \in \Z$, and $T_{\gamma}^n \in \Gamma$}\}.$$
If $\JKer_{g,b}^p < \Gamma$, then $K(\Gamma) = \Gamma \cap \JKer_{g,b}^p$, but the converse
does not hold.  Our second result is the following.

\begin{maintheorem}[Subgroups containing large pieces of Johnson kernel]
\label{theorem:jkervanish}
For some $g \geq 3$, let $\Gamma<\Mod(\Sigma_{g,b}^p)$ satisfy $[\Mod(\Sigma_{g,b}^p):\Gamma]<\infty$.
Assume that $[\Gamma \cap \JKer_{g,b}^n:K(\Gamma)] < \infty$.  Then $\HH_1(\Gamma;\R)=0$.
\end{maintheorem}

As a corollary, we obtain the following result, which was recently proven by Boggi \cite{Boggi} via
a difficult algebro-geometric argument under the assumption $b=p=0$.

\begin{maincorollary}[Subgroups containing Johnson kernel]
For some $g \geq 3$, let $\Gamma<\Mod(\Sigma_{g,b}^p)$ satisfy $[\Mod(\Sigma_{g,b}^p):\Gamma]<\infty$.
Assume that $\JKer_{g,b}^n < \Gamma$.  Then $\HH_1(\Gamma;\R)=0$.
\end{maincorollary}

\begin{remark}
McCarthy \cite{McCarthyCofinite} proved that Ivanov's conjecture fails in the case $g=2$.
\end{remark}

\paragraph{Acknowledgments.}
I wish to thank Martin Bridson, Benson Farb, Thomas Koberda, Dan Margalit, and Ben Wieland for useful 
comments and conversations.  
I also wish to thank Dongping Zhuang for showing me how to slightly weaken the hypotheses in my
original version of Theorem \ref{theorem:jkervanish}.

\section{Notation and basic facts about group homology}

If $M$ is a $G$-module, then $M_G$ will denote the {\em coinvariants} of the action, that is, the
quotient of $M$ by the submodule generated by the set $\{\text{$x - g(x)$ $|$ $x \in M$, $g \in G$}\}$.  This
appears in the 5-term exact sequence \cite[Corollary VII.6.4]{BrownCohomology}, which asserts the following.  If
$$\begin{CD}
1 @>>> K @>>> G @>>> Q @>>> 1
\end{CD}$$
is a short exact sequence of groups, then for any ring $R$, there is an exact sequence
$$\begin{CD}
\HH_2(G;R) @>>> \HH_2(Q;R) @>>> (\HH_1(K;R))_Q @>>> \HH_1(G;R) @>>> \HH_1(Q;R) @>>> 0.
\end{CD}$$

If $G_2 < G_1$ are groups satisfying $[G_1:G_2] < \infty$ and $R$ is a ring, then for all $k$ there exists a {\em transfer
map} of the form $t : \HH_k(G_1;R) \rightarrow \HH_k(G_2;R)$
(see, e.g., \cite[Chapter III.9]{BrownCohomology}).
The key property of $t$ (see \cite[Proposition III.9.5]{BrownCohomology}) 
is that if $i : \HH_k(G_2;R) \rightarrow \HH_k(G_1;R)$ is the map induced
by the inclusion, then $i \circ t : \HH_k(G_1;R) \rightarrow \HH_k(G_1;R)$ is multiplication by  $[G_1:G_2]$.
In particular, if $R = \R$, then we obtain a right inverse $\frac{1}{[G_1:G_2]} t$ to $i$.  This yields
the following standard lemma.

\begin{lemma}
\label{lemma:transfer}
Let $G_2 < G_1$ be groups satisfying $[G_1:G_2] < \infty$.  For all $k$, the
map $\HH_k(G_2;\R) \rightarrow \HH_k(G_1;\R)$ is surjective.
\end{lemma}

\section{Proof of Theorem \ref{theorem:dehntwistsvanish}}

Let $n \geq 1$ be the smallest integer such that $T^n_{\gamma} \in \Gamma$.

We first claim that there exists a subsurface $S \hookrightarrow \Sigma_{g,b}^p$ whose genus
is at least $2$ with the following property.  Let $i : \Mod(S) \rightarrow \Mod(\Sigma_{g,b}^p)$
be the induced map (``extend by the identity'').  Then there exists some boundary component $\beta$
of $S$ such that $i(T_{\beta}) = T_{\gamma}$.
There are two cases.  If $\gamma$ is nonseparating, then let $S$ be the complement
of a regular neighborhood of $\gamma$.  Observe that $S \cong \Sigma_{g-1,b+2}^p$, so the genus
of $S$ is at least $2$.  If instead $\gamma$ is separating, then let $S$ be the component of $\Sigma_{g,b}^p$
cut along $\gamma$ whose genus is maximal.  Since $g \geq 3$, this subsurface must have genus at least $2$.  The
claim follows.

Define $\Gamma' = i^{-1}(\Gamma)$.  We have $T_{\beta}^n \in \Gamma'$, and it
is enough to show that $[T_{\beta}^n]_{\Gamma'} = 0$.  Let $\overline{S}$ be the result of gluing
a punctured disc to $\beta$ and let $\overline{\Gamma}'$ be the image of $\Gamma'$ in $\Mod(\overline{S})$. 
There is a diagram of central extensions
$$\begin{CD}
1 @>>> \Z             @>>> \Gamma' @>>> \overline{\Gamma}' @>>> 1 \\
@.     @VV{\times n}V      @VVV         @VVV                    @. \\
1 @>>> \Z             @>>> \Mod(S) @>>> \Mod(\overline{S}) @>>> 1
\end{CD}$$
with $\Z < \Mod(S)$ and $\Z < \Gamma'$ generated by $T_{\beta}$ and $T_{\beta}^n$, respectively.  
The last 4 terms of the corresponding diagram of 5-term exact sequences are 
$$\begin{CD}
\HH_2(\overline{\Gamma}';\R) @>{f_1}>> \R             @>>> \HH_1(\Gamma';\R) @>>> \HH_1(\overline{\Gamma}';\R) @>>> 0 \\
@VV{f_2}V                              @VV{\cong}V         @VVV                   @VVV                              @.\\
\HH_2(\Mod(\overline{S});\R) @>{f_3}>> \R             @>>> \HH_1(\Mod(S);\R) @>>> \HH_1(\Mod(\overline{S});\R) @>>> 0
\end{CD}$$
We remark that there are no nontrivial coinvariants in these sequences since our extensions are central.
We must show that $f_1$ is a surjection.  Since $S$ has genus at least 2, we have
$\HH_1(\Mod(S);\R)=0$ (see, e.g., \cite{KorkmazSurvey}), so $f_3$ is a surjection.
Since $[\Mod(\overline{S}):\overline{\Gamma}'] < \infty$, 
Lemma \ref{lemma:transfer} implies that $f_2$ is a surjection, so
$f_1$ is a surjection, as desired.

\section{Proof of Theorem \ref{theorem:jkervanish}}

\subsection{Two facts about $\Sp_{2g}(\Z)$}

We will need two standard facts about finite-index subgroups $\Gamma$ of $\Sp_{2g}(\Z)$, both
of which follow from the fact that $\Gamma$ is a lattice in $\Sp_{2g}(\R)$.

For the first, since $\Sp_{2g}(\R)$ is a connected simple Lie group with finite center and real rank $g$, the group
$\Gamma$ has Kazhdan's property (T) when $g \geq 2$ (see, e.g., \cite[Theorem 7.1.4]{ZimmerBook}).  One standard
property of groups with property (T) is that they have no nontrivial homomorphisms to $\R$
(see, e.g., \cite[Theorem 7.1.7]{ZimmerBook}).  Combining these facts, we obtain the following theorem.

\begin{theorem}
\label{theorem:propertyt}
For some $g \geq 2$, let $\Gamma < \Sp_{2g}(\Z)$ satisfy $[\Sp_{2g}(\Z):\Gamma] < \infty$.  Then
$\HH_1(\Gamma;\R) = 0$.
\end{theorem}

For the second, since $\Sp_{2g}(\R)$ is a connected noncompact simple real algebraic group, we can apply
the Borel density theorem (see, e.g., \cite[Theorem 3.2.5]{ZimmerBook}) to deduce that $\Gamma$ is
Zariski dense in $\Sp_{2g}(\R)$.  This implies that any finite dimensional
nontrivial irreducible $\Sp_{2g}(\R)$-representation $V$ must
also be an irreducible $\Gamma$-representation; indeed, if $V'$ was a nontrivial proper $\Gamma$-submodule of $V$, 
then the subgroup of $\Sp_{2g}(\R)$ preserving $V'$ would be a proper subvariety 
of $\Sp_{2g}(\R)$ containing $\Gamma$.  Recall that the ring of coinvariants $V_{\Gamma}$ of $V$ under $\Gamma$ is the
quotient $V / K$, where $K = \Span{\text{$x-g(x)$ $|$ $x \in V$, $g \in \Gamma$}}$.  Since $K \neq 0$, we can apply
Schur's lemma to deduce that $K = V$, i.e.\ that $V_{\Gamma} = 0$.  We record this fact as the following theorem.

\begin{theorem}
\label{theorem:borelstab}
For some $g \geq 1$, let $\Gamma < \Sp_{2g}(\Z)$ satisfy $[\Sp_{2g}(\Z):\Gamma] < \infty$ and let
$V$ be a nontrivial irreducible $\Sp_{2g}(\R)$-representation.  Then $V_{\Gamma} = 0$.
\end{theorem}

\subsection{Two preliminary lemmas}

We will need two lemmas.  The first is the following, which
slightly generalizes a theorem of Johnson \cite{JohnsonKg}.

\begin{lemma}
\label{lemma:torellimodkg}
For $g \geq 3$, we have $\Torelli_{g,b}^p / \JKer_{g,b}^p \cong (\wedge^3 H)/H \oplus H^{b+p}$, 
where $H = \HH_1(\Sigma_g;\Z)$.
\end{lemma}
\begin{proof}
Since $\JKer_{g,b}^p$ contains all twists about boundary curves, we can assume that $b=0$.

Building on work of Johnson \cite{JohnsonAbel}, Hain \cite{HainTorelli} proved
that
$$\HH_1(\Torelli_{g}^p;\R) \cong (\wedge^3 H_{\R})/H_{\R} \oplus H_{\R}^{p},$$
where $H_{\R} = \HH_1(\Sigma_{g};\R)$.
Also, Johnson \cite[Lemma 2]{JohnsonAbel} proved that for $x \in \JKer_{g}^p$, we have 
$[x]_{\Torelli_{g}^p} = 0$ (Johnson only considered
the case where $p=0$, but his argument works in general).  It follows that
\begin{equation}
\label{eqn:abelianization}
\HH_1(\Torelli_{g}^p / \JKer_{g}^p;\R) \cong (\wedge^3 H_{\R})/H_{\R} \oplus H_{\R}^{p}.
\end{equation}

We will prove the lemma by induction on $p$.  The base case $p=0$ is a theorem of Johnson \cite{JohnsonKg}.
Assume now that $p > 0$ and that the lemma is true for all smaller $p$.
Fixing a puncture $\ast$ of $\Sigma_{g}^p$,
work of Birman \cite{BirmanSeq} and Johnson \cite{JohnsonAbel} gives
an exact sequence
$$\begin{CD}
1 @>>> \pi_1(\Sigma_{g}^{p-1},\ast) @>>> \Torelli_{g}^p @>>> \Torelli_{g}^{p-1} @>>> 1,
\end{CD}$$
where the map $\Torelli_{g}^p \rightarrow \Torelli_{g}^{p-1}$ comes from ``forgetting the puncture $\ast$''.
Quotienting out by $\JKer_{g}^p$, we obtain an exact sequence
$$\begin{CD}
1 @>>> \pi_1(\Sigma_{g}^{p-1},\ast) / (\pi_1(\Sigma_{g}^{p-1},\ast) \cap \JKer_{g}^p) @>>> \Torelli_{g}^p / \JKer_{g}^p @>>> \Torelli_{g}^{p-1} / \JKer_{g}^{p-1} @>>> 1.
\end{CD}$$
By induction, we have
$$\Torelli_{g}^{p-1} / \JKer_{g}^{p-1} \cong (\wedge^3 H)/H \oplus H^{p-1}.$$
Set $A = \pi_1(\Sigma_{g}^{p-1},\ast) / (\pi_1(\Sigma_{g}^{p-1},\ast) \cap \JKer_{g}^p)$.  
We will prove that $A$
is a quotient of $H$.  We will then be able to conclude that 
$\Torelli_{g}^{p-1} / \JKer_{g}^{p-1}$ acts trivially on $A$,
so $\Torelli_{g}^p / \JKer_{g}^p$ is the abelian group
$$(\wedge^3 H)/H \oplus H^{p-1} \oplus A.$$
Using \eqref{eqn:abelianization}, a simple dimension
count will then imply that $A$ cannot
be a proper quotient of $H$, and the lemma will follow.

\Figure{figure:birmanexactsequence}{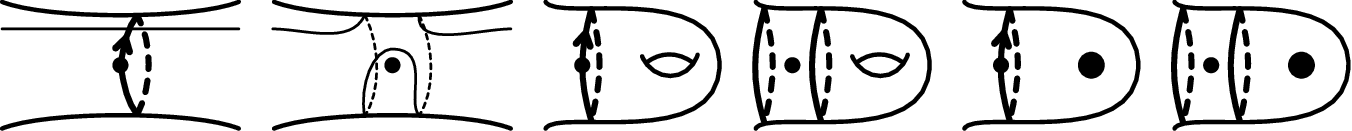}{a--f. Curves needed for proof of Lemma \ref{lemma:torellimodkg}}

The element of $\Torelli_{g}^p$ corresponding to $\delta \in \pi_1(\Sigma_{g}^{p-1},\ast)$ ``drags'' $\ast$
around $\delta$.  As shown in Figures \ref{figure:birmanexactsequence}.a--b, a simple
closed curve $\gamma \in \pi_1(\Sigma_{g}^{p-1},\ast)$
corresponds to $T_{\gamma_1} T_{\gamma_2}^{-1} \in \Torelli_{g}^p$, where $\gamma_1$ and $\gamma_2$ are the
boundary components of a regular neighborhood of $\gamma$.  In particular, if $\gamma$ is a simple closed separating
curve, then as shown in Figures \ref{figure:birmanexactsequence}.c--d, the corresponding element of
$\Torelli_{g}^p$ is a product of separating twists.  Since $[\pi_1(\Sigma_{g}^{p-1},\ast),\pi_1(\Sigma_{g}^{p-1},\ast)]$
is generated by simple closed separating curves (see, e.g.,\ \cite[Lemma A.1]{PutmanCutPaste}), we deduce 
that $[\pi_1(\Sigma_{g}^{p-1},\ast),\pi_1(\Sigma_{g}^{p-1},\ast)] \subset \pi_1(\Sigma_{g}^{p-1},\ast) \cap \JKer_{g}^p$.
Thus $A=\pi_1(\Sigma_{g}^{p-1},\ast) / (\pi_1(\Sigma_{g}^{p-1},\ast) \cap \JKer_{g}^p)$ is a quotient of
$\HH_1(\Sigma_{g}^{p-1};\Z)$.
Finally, as shown in Figures \ref{figure:birmanexactsequence}.e--f, all simple closed curves that are
homotopic into punctures are also contained in $\pi_1(\Sigma_{g}^{p-1},\ast) \cap \JKer_{g}^p$, so we conclude
that $A$ is a quotient of $H=\HH_1(\Sigma_{g};\Z)$,
as desired.
\end{proof}

For the second lemma, define $Q_{g,b}^p = \Mod_{g,b}^p / \JKer_{g,b}^p$.  

\begin{lemma}
\label{lemma:qdies}
For some $g \geq 3$, let $Q'<Q_{g,b}^p$ satisfy $[Q_{b,b}^p:Q'] < \infty$.  Then
$\HH_1(Q';\R) = 0$.
\end{lemma}
\begin{proof}
Restricting the short exact sequence
$$\begin{CD}
1 @>>> \Torelli_{g,b}^p / \JKer_{g,b}^p @>>> Q_{b,b}^p @>>> \Sp_{2g}(\Z) @>>> 1
\end{CD}$$
to $Q'$, we obtain a short exact sequence
$$\begin{CD}
1 @>>> B @>>> Q' @>>> \overline{Q}' @>>> 1,
\end{CD}$$
where $B$ and $\overline{Q}'$ are finite index subgroups of $\Torelli_{g,b}^p / \JKer_{g,b}^p$ and
$\Sp_{2g}(\Z)$, respectively.  The last 3 terms of the associated 5-term exact sequence are
$$\begin{CD}
(\HH_1(B;\R))_{\overline{Q}'} @>>> \HH_1(Q';\R) @>>> \HH_1(\overline{Q}';\R) @>>> 0.
\end{CD}$$
By Theorem \ref{theorem:propertyt}, we have $\HH_1(\overline{Q}';\R) = 0$.  Letting 
$H = \HH_1(\Sigma_g;\Z)$, Lemma \ref{lemma:torellimodkg} says
that 
$$\Torelli_{g,b}^p / \JKer_{g,b}^p \cong (\wedge^3 H)/H \oplus H^{b+p}.$$  
Since $B$ is a finite-index subgroup of $\Torelli_{g,b}^p / \JKer_{g,b}^p$, we get that $B$ is itself abelian and
$$\HH_1(B;\R) \cong B \otimes \R \cong (\Torelli_{g,b}^p / \JKer_{g,b}^p) \otimes \R \cong (\wedge^3 H_{\R})/H_{\R} \oplus H_{\R}^{b+p},$$
where $H_{\R} = \HH_1(\Sigma_g;\R)$.  
Both $(\wedge^3 H_{\R})/H_{\R}$ and $H_{\R}$ are nontrivial finite-dimensional irreducible representations of $\Sp_{2g}(\R)$, so
Theorem \ref{theorem:borelstab} implies that $(\HH_1(B;\R))_{\overline{Q}'} = 0$, and we are done.
\end{proof}

\subsection{The proof of Theorem \ref{theorem:jkervanish}}

The last 3 terms of the 5-term exact sequence associated to the
short exact sequence
$$\begin{CD}
1 @>>> \Gamma \cap \JKer_{g,b}^p @>>> \Gamma @>>> \Gamma / (\Gamma \cap \JKer_{g,b}^p) @>>> 1
\end{CD}$$
are
$$\begin{CD}
(\HH_1(\Gamma \cap \JKer_{g,b}^p;\R))_{\Gamma / (\Gamma \cap \JKer_{g,b}^p)} @>{i}>> \HH_1(\Gamma;\R) @>>> \HH_1(\Gamma / (\Gamma \cap \JKer_{g,b}^p);\R) @>>> 0.
\end{CD}$$
By assumption, $[\Gamma \cap \JKer_{g,b}^p:K(\Gamma)] < \infty$, so Lemma \ref{lemma:transfer}
implies that the map $\HH_1(K(\Gamma);\R) \rightarrow \HH_1(\Gamma \cap \JKer_{g,b}^p;\R)$ is surjective.
Since $K(\Gamma)$ is generated by powers of twists, Theorem
\ref{theorem:dehntwistsvanish} allows us to deduce that $i=0$.  Also, $\Gamma / (\Gamma \cap \JKer_{g,b}^p)$ is a finite-index
subgroup of $Q_{g,b}^p$, so Lemma \ref{lemma:qdies} implies that $\HH_1(\Gamma / (\Gamma \cap \JKer_{g,b}^p);\R)=0$, and
we are done.

\noindent
Department of Mathematics; MIT, 2-306 \\
77 Massachusetts Avenue \\
Cambridge, MA 02139-4307 \\
E-mail: {\tt andyp@math.mit.edu}
\medskip

\end{document}